\documentclass[12pt]{amsart}

\usepackage{pb-diagram}

\usepackage{amsfonts}
\usepackage{amsmath}
\usepackage{amssymb}

\newcommand{\dbar}{\overline{\partial}}

\newcommand{\ddt}[1]{\frac{\partial #1}{\partial t}}

\newcommand{\ddbar}{\sqrt{-1}\partial\dbar}

\newtheorem{theorem}{Theorem}[section]
\newtheorem{proposition}{Proposition}[section]
\newtheorem{lemma}{Lemma}[section]
\newtheorem{conjecture}{Conjecture}[section]

\newtheorem{corollary}{Corollary}[section]

\numberwithin{equation}{section}

\begin{document}

\title[Collapsing behavior of Ricci-flat K\"ahler metrics]{Collapsing behavior of Ricci-flat K\"ahler metrics and long time solutions of the K\"ahler-Ricci flow}

\author[Jian Song, Gang Tian and Zhenlei Zhang]{Jian Song$^*$, Gang Tian$^{**}$, Zhenlei Zhang$^\dagger$}

\address{$^*$ Department of Mathematics, Rutgers University, Piscataway, NJ 08854}

\address{$^{**}$ Department of Mathematics, Peking University, Beijing }
\address{$^\dagger$Department of Mathematics,  Capital Normal University, Beijing
}

\thanks{Research supported in
part by NSF DMS-1711439, NSFC11331001, NSFC(11825105, 11821101, 11771301) and Fok Foundation 161001.}

\begin{abstract} We prove a uniform diameter bound for long time solutions of the normalized K\"ahler-Ricci flow on an $n$-dimensional projective manifold $X$ with semi-ample canonical bundle under the assumption that the Ricci curvature is uniformly bounded for all time in a fixed domain containing a fibre of $X$ over its canonical model $X_{can}$. This assumption on the Ricci curvature always holds when the Kodaira dimension of $X$ is $n$, $n-1$ or when the general fibre of $X$ over its canonical model is a complex torus. In particular,  the normalized K\"ahler-Ricci flow converges in Gromov-Hausdorff topolopy to its canonical model when $X$ has Kodaira dimension $1$ with $K_X$ being semi-ample and the general fibre of $X$ over its canonical model being a complex torus. We also prove the Gromov-Hausdorff limit of collapsing Ricci-flat K\"ahler metrics on a holomorphically fibred Calabi-Yau manifold is unique and is homeomorphic to the metric completion of the corresponding twisted K\"ahler-Einstein metric on the regular part of its base.

\end{abstract}

\maketitle

\section{Introduction}

Hamilton introduced the Ricci flow to study the global structures and classification of Riemannian manifolds in his seminal work \cite{H1}. It was used to solve Thurston's geometrization conjecture for $3$-manifolds \cite{P1, P2, P3, KL, MT}. The Ricci flow preserves the K\"ahlerian metrics: If $(X,g_0)$ is compact K\"ahler manifold of complex dimension $n$, then any solution $g(t)$ of the Ricci flow with initial condition $g_0$ must be K\"ahler. This leads to the K\"ahler-Ricci flow which has been very useful in K\"ahler geometry: 
 \begin{equation}\label{unkrflow}
\left\{
\begin{array}{l}
{ \displaystyle \ddt{g} = -Ric(g) ,}\\
\\
g|_{t=0} =g_0 .
\end{array} \right.
\end{equation}
In \cite{Cao1}, adapting certain arguments in \cite{Y1}, Cao first studied the K\"ahler-Ricci flow and used this to give an alternative proof of  the Calabi conjecture.
 
The Ricci flow always has a solution for $t$ small. It was proved in \cite{TZha} that \eqref{unkrflow} admits a maximal solution on $[0, T)$, where 
\begin{equation} \label{T}
T \,=\, \sup \{ t \in \mathbb{R} \ | \ [\omega_0] + t  [K_X] >0 \}.
\end{equation}
where $\omega_0$ denotes the K\"ahler form of $g_0$. This gives a sharp local existence theorem. Therefore, the K\"ahler-Ricci flow admits a long time solution if and only if the canonical class is nef.  It was also shown in \cite{Cao1} that the K\"ahler-Ricci flow after normalization always converges exponentially fast to a K\"ahler-Einstein metric if the first Chern class is negative or zero. If the first Chern class is positive, $T$ is finite and one can study finer behavior of $g(t)$ as $t$ tends to $T$. This has been extensively studied (see \cite{P4, SeT, TZhu1,  PSSW, CS, TiZ1, B, CW} etc.). One would hope that the K\"ahler-Ricci flow should deform any initial K\"ahler metric to a K\"ahler-Einstein metric, however, most K\"ahler manifolds do not admit definite or vanishing first Chern class and so the flow will in general develop singularities. An Analytic Minimal Model Program (AMMP) through Ricci flow was initiated by Song-Tian more than a decade ago to study birational classification of compact K\"ahler manifolds, including algebraic manifolds. We refer the readers to \cite{SoT3} for a description of the AMMP.
One crucial problem in this program is to study formation of singularities along the K\"ahler-Ricci flow. It is conjectured by Song-Tian in \cite{SoT3} that the K\"ahler-Ricci flow will deform a projective variety $X$ of nonnegative Kodaira dimension, to its minimal model via finitely many divisorial metric contractions and metric flips in Gromov-Hausdorff topology, then eventually converge to a unique canonical metric of Einstein type on its unique canonical model. The existence and uniqueness is proved in \cite{SoT3} for the analytic solutions of the K\"ahler-Ricci flow on algebraic varieties with log terminal singularities and the K\"ahler-Ricci flow can be analytically and uniquely extended through divisorial contractions and flips \cite{SoT3}. Finite time geometric surgeries in terms of Gromov-Hausdorff topology are introduced and established in the case of K\"ahler surfaces and more generally, flips induced by Mumford quotients in \cite{SW1, SW2, SW3, SY, S1}. An alternative approach to understand the K\"ahler-Ricci flow through singularities was proposed in \cite{LT} in the frame work of K\"ahler quotients by transforming the parabolic complex Monge-Amp\`ere equation into an elliptic  $V$-soliton equation.

In this paper, we are interested in geometric behavior of long time solutions of the K\"ahler-Ricci flow. We consider the normalized K\"ahler-Ricci flow on an $n$-dimensional projective manifold $X$ defined by
 \begin{equation}\label{krflow}
\left\{
\begin{array}{l}
{ \displaystyle \ddt{g} = -Ric(g) - g,}\\
\\
g|_{t=0} =g_0 .
\end{array} \right.
\end{equation}
with the initial K\"ahler metric $g_0$. The long time existence of the flow (\ref{krflow}) is equivalent to the canonical class $K_X$ being nef. A projective manifold of nef canonical model is called a minimal model. The well-known abundance conjecture in birational geometry predicts that  $K_X$ being nef is equivalent to $K_X$ being semi-ample and the conjecture holds up to complex dimension $3$. In this paper we will always assume the canonical bundle $K_X$ to be semi-ample, and a uniform scalar curvature bounded is established for the global solutions of the flow (\ref{krflow}).  

The Kodaira dimension is the algebraic dimension measuring the size of the pluricanonical system of the underlying K\"ahler manifold $X$. When the canonical bundle $K_X$ is semi-ample or nef, the Kodaira dimension of $X$ must be a nonnegative integer no greater than $\dim_\mathbb{C} X=n$.
We will discuss some of the known results by Kodaira dimension ${\rm Kod}(X)$.

 When the ${\rm Kod}(X) =n$,  the canonical class of $X$ is big and nef and it was proved in \cite{Ts, TZha} that  the flow (\ref{krflow})  converges weakly to the unique K\"ahler-Einstein current on the canonical model $X_{can}$ of $X$ which is smooth on $X_{can}^\circ$, the regular set of $X_{can}$. Recently, it was proved in \cite{S2} that the metric completion of such a smooth K\"ahler-Einstein metric on $X_{can}$ is homeomorphic to the projective variety $X_{can}$ itself. This result was used in \cite{W} to obtain a uniform diameter bound for the flow (\ref{krflow}). In the special case when $K_X$ is ample,  it was shown in \cite{Cao1} that the flow will always converge smoothly to the unique K\"ahler-Einstein metric on $X$.

When ${\rm Kod}(X)=0$, $X$ must a Calabi-Yau manifold, as we mentioned above, the unnormalized flow (\ref{unkrflow}) converges smoothly \cite{Cao1} to the unique Ricci-flat K\"ahler metric in the initial K\"ahler class.

When  ${\rm Kod} (X)=\kappa$ and $1\leq \kappa \leq n-1$,  for sufficiently large $m$, the canonical ring is finitely generated and the pluricanonical system $|mK_X|$ induces a unique projective morphism
$$\Phi: X \rightarrow X_{can}$$
where $X_{can}$ is the unique canonical model of $X$ and $\dim X_{can} = \kappa$. $X_{can}$ has canonical singularities and the map $\Phi$ is a holomorphic fibration of $(n-\kappa)$-dimensional manifolds of vanishing first Chern class over $X_{can}^\circ$, the set of smooth points on $X_{can}$ over which $\Phi$ is submersion.  $X_{can}^\circ$ is a Zariski open set of $X_{can}$. In \cite{ST1,ST2}, the twisted K\"ahler-Einstein (possibly singular) metric $g_{can}$ on $X_{can}$ is defined to be
\begin{equation}\label{twke}
Ric(g_{can}) = - g_{can} + g_{WP},
\end{equation}
where $g_{WP}$ is the Weil-Petersson current induced from the variation of Calabi-Yau fibres of $X$ over $X_{can}$. In fact, $g_{can}$ has bounded local potentials and it is smooth on $X_{can}^\circ$.  It was shown in \cite{ST1, ST2} that the solution of (\ref{krflow}) converges in current to the canonical twisted K\"ahler-Einstein current $g_{can}$ on $X_{can}$ and the local potential converges in $C^{1,\alpha}$ on any compact subset of $X^\circ= \Phi^{-1}(X_{can}^\circ)$. The local convergence is improved to the $C^0$-topology for $g(t)$ to $g_{can}$ in \cite{TWY}. 

We now state our first result of the paper.

\begin{theorem} \label{main1} Let $X$ be a projective manifold with semi-ample canonical bundle. Let $g(t)$ be the solution of the normalized K\"ahler-Ricci flow (\ref{krflow}) on $X$ with any initial K\"ahler metric $g_0$. If there exists an open domain $U$ of $X$ containing a fibre of $X$ over $X_{can}$ and $\Lambda>0$  such that 
\begin{equation}\label{curvature bound}
\sup_{U\times [0, \infty)} |Ric(g(t))|_{g(t)} <\Lambda .
\end{equation} Then there exists $D>0$ such that for all $t\in [0, \infty)$ we have
\begin{equation}\label{diambd}
diam(X, g(t)) < D ,
\end{equation}
where $diam(X, g(t))$ is the diameter of $(X, g(t))$.
\end{theorem}

The assumption on the Ricci curvature in Theorem \ref{main1} is used to apply the relative volume comparison for the Ricci flow proved in \cite{TiZ3}. It is expected that such a Ricci bound holds on any open domain $U$ compactly supported  in $X^\circ= \Phi^{-1}(X_{can}^\circ)$,  but it is still open at the moment. The recent work of \cite{HT} might give some technical hints on how to bound the Ricci curvature locally in $X^\circ$. In fact, when the general fibre of $\Phi: X\rightarrow X_{can}$ is a complex torus, then the full curvature tensors are uniformly bounded on any compact subset of $X_{can}^\circ$ \cite{TWY} and the following corollary immediately follows from Theorem \ref{main1}. Combined with the uniform bounds for the scalar curvature
$$\sup_{X\times [0, \infty)} |R(g(t))| < \infty$$
for all time (\cite{SoT4}), the estimate in Theorem \ref{main1} for long time collapsing solutions can be compared to Perelman's diameter and scalar curvature bounds for non-collapsed solutions of the K\"ahler-Ricci flow on Fano manifolds. The diameter estimate can also be viewed as a boundedness result in view of algebraic geometry and topology. Theorem \ref{main1} can be largely extended to the K\"ahler case since the semi-ampleness of the canonical bundle already implies that the canonical model is projective.

\begin{corollary}\label{mainc1} Let $g(t)$ be the solution of the normalized K\"ahler-Ricci flow (\ref{krflow}) on $X$ with any initial K\"ahler metric $g_0$. If ${\rm Kod}(X)\geq \dim X - 1$,  or more generally if the general fibre of $\Phi: X \rightarrow X_{can}$ is a complex torus, then there exists $D>0$ such that for all $t\in [0, \infty)$ we have
$$diam(X, g(t)) < D .$$

\end{corollary}

By Theorem \ref{main1}, one can always extract a convergent sequence  along time for the solution of the normalized K\"ahler-Ricci flow in Gromov-Hausdorff topology, however, more delicate analysis is required to show the uniqueness of such limits. We naturally propose the following conjecture about the convergence of the flow (\ref{krflow}) on smooth minimal models, i.e. projective manifolds with nef canonical bundle.

\begin{conjecture} \label{conj} Let $X$ be a smooth K\"ahler manifold with nef canonical bundle. Then for any initial K\"ahler metric $g_0$ on $X$, the solution $g(t)$ of the unnormalized K\"ahler-Ricci flow (\ref{krflow}) has uniformly bounded diameter and scalar curvature for all time. Furthermore, $(X, g(t))$ converges in the Gromov-Hausdorff topology to a unique compact metric space $(\mathcal{Z}, d_\mathcal{Z}) $ homeomorphic to the canonical $X_{can}$.

\end{conjecture}
 We confirm the conjecture in the following theorem  when $K_X$ is semi-ample, the Kodaira dimension is $1$and the general fibre of $X$ over $X_{can}$ is a complex torus. In particular, it confirms the conjecture for the K\"ahler-Ricci flow on minimal elliptic surfaces of Kodaira dimension $1$ in \cite{ST1}.

\begin{theorem}\label{KRF: minimal model}
Under the same assumptions including (\ref{curvature bound}) in Theorem \ref{main1}, if the projective manifold $X$ has Kodaira dimension $1$, then any solution of the normalized K\"ahler-Ricci flow (\ref{krflow}) converges in the Gromov-Hausdorff topology to the metric completion of  $(X_{can}^\circ, g_{can})$, which is homeomorphic to $X_{can}$. In particular, the conclusion holds for projective manifolds of Kodaira dimension $1$ and semi-ample canonical bundle whose general fibre over its canonical model is a complex torus. 
\end{theorem}

We remark that the result in Theorem \ref{KRF: minimal model} still holds in the case of higher Kodaira dimension if $X_{can}\setminus X_{can}^\circ$  is a set of finitely many isolated points.

The proof of Theorem \ref{main1} relies on the diameter estimate for certain family of twisted K\"ahler-Einstein metrics established in \cite{FGS} and the relative volume comparison established in \cite{TiZ3}. The main technical contribution of the paper is to prove that the evolving metrics $g(t)$ of the K\"ahler- Ricci flow has suitable convexity on $\Phi^{-1}(X_{can}^\circ)$. Such a convexity result is built on the almost  convexity of the twisted K\"ahler-Einstein metric $g_{can}$ on $X_{can}^\circ$ in the following theorem for the continuity method improving the results in  \cite{FGS}.

\begin{theorem} \label{main2}Let $X$ be a projective K\"ahler manifold of $\dim_{\mathbb{C}} X =n$ with semi-ample canonical line bundle $K_X$ with $X_{can}$ being the canonical model of $X$. Let $A$ be an ample line bundle and $g_t\in [tA + K_X]$ be the unique K\"ahler metrics for $t\in (0, \infty)$ satisfying
$$Ric(g_t) = -g_t + t g_A, $$
for any fixed K\"ahler metric $\omega_A\in [A]$ on $X$. Then the following hold.

\begin{enumerate}

\item $(X, g_t)$ converges in the Gromov-Hausdorff topology to a compact metric space $(\mathcal{Z}, d_\mathcal{Z})$ as $t\rightarrow 0^+$.

\item  $g_t$ converges in the $C^0$-topology on $\Phi^{-1}(X_{can}^\circ)$ to the pullback of $g_{can}$ on $X_{can}^\circ$ as $t\rightarrow 0^+$.

\item The metric completion of $(X_{can}^\circ, g_{can})$ is isomorphic to $(\mathcal{Z}, d_\mathcal{Z})$.

\end{enumerate}
In particular, if $\dim_{\mathbb{C}} X_{can} \leq 2$, or more generally if $X_{can}$ has only orbifold singularities, $(\mathcal{Z}, d_\mathcal{Z})$ is homeomorphic to $X_{can}$.

\end{theorem}

The advantage of using the continuity method over the Ricci flow is that the Ricci curvature is naturally bounded below uniformly and one can apply many existing techniques in comparison geometry.   A more natural adaption of Theorem \ref{main2} is for the collapsing behavior of Ricci-flat K\"ahler metrics on a Calabi-Yau manifold as a holomorphic fibration of Calabi-Yau manifolds. This topic has been extensively studied in \cite{Tos, GTZ, ToZh, HT}.

\begin{theorem} \label{main3}Let $X$ be a projective K\"ahler manifold of $\dim_{\mathbb{C}} X =n$ with $c_1(X)=0$. Suppose $L$ is a semi-ample line bundle over $X$. The linear system $H^0(X, mL)$ induces a holomorphic map $\Phi: X \rightarrow Y$ for sufficiently large $m$ from $X$ to  a projective variety $Y$. Let $A$ be an ample line bundle and $g_t\in [tA + L]$ be the unique Calabi-Yau metrics for $t\in (0, \infty)$. Then the followings hold.

\begin{enumerate}

\item $(X, g_t)$ converges in the Gromov-Hausdorff topology to a compact metric space $(\mathcal{Z}, d_\mathcal{Z})$ as $t\rightarrow 0^+$,

\item  $g_t$ converges in the $C^0$-topology on $\Phi^{-1}(Y^\circ)$ to the pullback of a unique smooth K\"ahler metric $g_Y$ on $Y^\circ$ as $t\rightarrow 0^+$.  Here $Y^\circ$ is the set of smooth points of $Y$ over which $\Phi$ is submersion. 

\item $g_Y$ extends to a K\"ahler current on $Y$ with bounded local potentials and on $Y^\circ$, we have
$$Ric(g_Y) = g_{WP},$$
where $g_{WP}$ is the Weil-Petersson metric of the variation of the smooth Calabi-Yau fibres of $X$ over $Y^\circ$.

\item The metric completion of $(Y^\circ, g_Y)$ is isomorphic to $(\mathcal{Z}, d_\mathcal{Z})$.

\end{enumerate}
In particular, if $\dim_{\mathbb{C}} Y \leq 2$ or more generally if $Y$ has only orbifold singularities, $(\mathcal{Z}, d_\mathcal{Z})$ is homeomorphic to $Y$.

\end{theorem}

The twisted Ricci-flat K\"ahler metric $g_Y$ in Theorem \ref{main3} was proposed in \cite{ST1, ST2} as a special case of the twisted K\"ahler-Einstein metrics and it had already been implicitly studied in \cite{GW, Fi} in the case of complex surfaces. The statement (4) in Theorem \ref{main3} confirms the conjecture proposed in \cite{ToZh} (Conjecture 1.1 (a) and (b)) related to an analogous conjecture by Gross \cite{Gr}, Kontsevich-Soilbelman \cite{KS} and Todorov \cite{Ma} for collapsing limits of Ricci-flat K\"ahler metrics near complex structure limits. The statement (3) in Theorem \ref{main3} is shown in \cite{ST1, ST2}. The  sequential convergence in statements (1) is proved in \cite{Tos, GTZ} and the statement (2) is proved in \cite{TWY}. The special case for the statement (4)  is proved in \cite{GTZ, ToZh} when $\dim_{\mathbb{C}} Y=1$ or $X$ is Hyperkahler. The main contribution of our work in this paper is the statement (4) for identifying the intrinsic and extrinsic geometric limits of collapsing Calabi-Yau metrics. 

We would also like to point out that the projective assumption for the K\"ahler manifold $X$ in this paper is for conveniences and is not essential in the proof. In fact, the semi-ampleness condition already implies $X_{can}$ is projective. 

We give a brief outline of the paper. In section 2, we prove Theorem \ref{main2} and in particular, we show that $(X_{can}^\circ, g_{can})$ is almost geodesically convex in $(\mathcal{Z}, d_\mathcal{Z})$. In section 3, Theorem \ref{main3} is proved by slight modification of the proof of Theorem \ref{main2}. In section 4, we prove our main result Theorem \ref{main1} and its corollaries by using the result and proof of section 2. Finally, we prove Theorem \ref{KRF: minimal model} in section 5. 

\section{Proof of Theorem \ref{main2}}

In this section, we will study deformation of a family of collapsing twisted K\"ahler-Einstein metrics and prove   Theorem \ref{main2}.

Let $X$ be a projective manifold of complex dimension $n$. Suppose the canonical line bundle $K_X$ is semi-ample and the Kodaira dimension of $X$ is $\kappa$, i.e., $\dim_{\mathbb{C}}X_{can} = \kappa$. In this section we will always assume $0<\kappa < n$. Then the pluricanonical system $|mK_X|$ induces a holomorphic morphism
$$\Phi: X \rightarrow X_{can}\hookrightarrow \mathbb{CP}^{N_m}, $$
for sufficiently large $m$ and $X_{can}$ is the unique canonical model of $X$.

Let $\omega_A$ be a K\"ahler metric in a K\"ahler class $A$ on $X$. We will now consider a continuous family of K\"ahler metrics $ \omega (t)$ defined by
\begin{equation}\label{contin1}
Ric(  \omega(t)) = -   \omega(t) + t \omega_A, ~ t\in (0, 1].
\end{equation}
We let $\Omega$ be a smooth volume form on $X$ such that
\begin{equation}
\chi = \ddbar\log \Omega =\frac{1}{m} \Phi^*\omega_{FS}\in [K_X],
\end{equation}
where $\omega_{FS}$ is the Fubini-Study metric of 
$\mathbb{CP}^{N_m}$. 
Therefore 
$$[  \omega(t)] = [\chi] + t [\omega_A].$$ Throughout the paper, we abuse the notation by identifying  $\chi$ with $\frac{1}{m}\omega_{FS}|_{X_{can}}$ on $X_{can}$ as well.
If we write $$  \omega(t) = \chi +t \omega_A +\ddbar \psi (t)$$
for some $\psi = \psi(t) \in C^\infty(X)$,
then by straightforward deductions, the equation (\ref{contin1}) becomes
\begin{equation}\label{contin2}
t^{-(n-\kappa)} (\chi + t \omega_A+ \ddbar \psi)^n = e^{\psi} \Omega.
\end{equation}

Equation (\ref{contin2}) has a smooth solution for all $t >  0$ by \cite{Y1, A} since $[\chi+t\omega_A ]$ is a K\"ahler class and we are interested in the the limiting behavior of $  \omega(t)$ as $t \rightarrow 0$.   We first state some basic estimates for $\psi$.

\begin{theorem} There exists $C>0$ such that for all $t\in (0,1]$, we have
$$|\psi| \leq C. $$

\end{theorem}

\begin{proof} By the maximum principle, there exist $C_1, C_2>0$ such that for any $t\in (0,1]$, we have
$$\sup_X \psi \leq \sup_X \left(\log \frac{  t^{-(n-\kappa)}(\chi + t \omega_A )^n}{\Omega} \right)\leq C_1  \sup_X \left(\log \frac{ \chi^\kappa \wedge \omega_A^{n-\kappa}}{\Omega} \right)\leq C_2. $$
Then the right hand side of the equation (\ref{contin2}) is uniformly bounded above and the lower bound of $\psi$ follows directly from $L^\infty$-estimate for degenerate complex Monge-Amp\`ere equations by Demailly-Pali \cite{DP} (see also \cite{Kol1, EGZ}).

\end{proof}

We notice that for any $t\in (0,1]$, the Ricci curvature of $  \omega(t)$ is uniformly bounded below by $-1$. We can then apply the following diameter estimate proved in \cite{FGS}.

\begin{lemma} Let $  g(t)$ be the K\"ahler metric associated to $  \omega(t)$.
There exists $L>0$ such that for all $t\in (0,1]$,
$$Diam(X,   g(t)) \leq L.$$

\end{lemma}

By the volume comparison, we immediately have the following volume estimate.

\begin{corollary} \label{volr} There exists $C>0$ such that for any point $p\in X$ and $t\in (0,1]$ and $0< r < L$,
$$Vol(B_{  g(t)} (p, r),   g(t) ) \geq Cr^{2n}  t^{n-\kappa}.$$

\end{corollary}

The following lemma is due to \cite{G} (also see \cite{CC, Da}) as a consequence of the volume comparison and it is very useful to prove geometric convexity for certain family of metric spaces.

\begin{lemma} \label{gromov} Let $(M, g)$ be a Riemannian manifold of dimension $n$ satisfying
$$Ric(g) \geq - g, ~ Diam(M, g) \leq L.$$
Let  $E\subset M$ be any compact set with a smooth boundary. If there are two points $p_1, p_2\in M$ with $$B_g(p_i, r) \cap E = \emptyset, ~ i =1, 2$$
and every minimal geodesic from $p_1$ to points in $B(p_2, r)$ intersects $E$, then there exists $c=c(n, r, L)>0$ such that
$$Vol(\partial E, g) \geq c Vol(B_g(p_2, r), g). $$

\end{lemma}

We will construct the set $E$ for the family of metrics $  \omega(t)$.
First, by semi-ampleness of $K_X$ we can assume $K_X$ is the pullback of an ample line bundle $\mathcal{L}$ on $X_{can}$. We can pick an effective $\mathbb{Q}$-divisor $\sigma$ on $X_{can}$ such that

\begin{enumerate}

\item $\sigma$ lies in the class of $[\mathcal{L}]$,

\item  $X_{can}\setminus X_{can}^\circ$ is contained in the support of $\sigma$.

\end{enumerate}
We let $\sigma'=\Phi^*\sigma$.

Second, we consider a log resolution of $X_{can}$ defined by
$$\Psi: W \rightarrow X_{can}$$
such that

\begin{enumerate}

\item $W$ is smooth and the exceptional locus of $\Psi$ is a union of smooth divisors of simple normal crossings.

\item $\tilde \sigma$, the pullback of $\sigma$, is a union of smooth divisors of simple normal crossings.

\end{enumerate}  The Fubini-Study metric $\chi$ on $X_{can}$ also lies in $[\sigma]$.
Let $\tilde X$ be the blow-up of $X$ induced by $\Psi: W \rightarrow X_{can}$ and we let $  \Psi':   \tilde X \rightarrow X$. We also pick the hermitian metric $h$ on $\mathcal{L}$ such that $Ric(h) = \chi$. Let $\tilde \sigma' = \Psi^* \sigma'$. Away from $\tilde \sigma'$, $\tilde X$ can be identified as $X$ by assuming the blow-ups take place at the support of $ \sigma'$. We also let $ h' = \Phi^*h$ and $\tilde h' = \Psi^* h'$.

The following is an analogue of the Schwarz lemma.
\begin{lemma} \label{schwa} There exists $c>0$ such that for all $t\in (0,1]$ we have on $X$,
$$  \omega(t) \geq c \chi. $$

\end{lemma}

\begin{proof} We can directly apply the maximum principle to the following quantity
$$\log tr_{ \omega}(\chi) - K \psi$$
for sufficiently large $K$ and the estimate of the lemma will immediately follows as $\psi$ is uniformly bounded.

\end{proof}

For simplicity, we assume that   
$$|\tilde \sigma'|^2_{\tilde h'} \leq 1$$
everywhere.
Let $F$ be the standard decreasing smooth cut-off function defined on $[0, \infty)$ satisfying

\begin{enumerate}
\item $F(x)=3$,  if $x\in [0, 1/2]$,

\item  $F(x)=0$,  if $x\in [3, \infty)$,
\item $F(x) = 3-x$, if $x\in [1,2]$.
\end{enumerate}
Let $$\eta_\epsilon = \max \left( \log |\tilde\sigma'|^2_{\tilde h'}, \log \epsilon \right)$$
for some sufficiently small $\epsilon>0$ to be determined later.
By the construction of $\tilde\sigma'$ and $h$, we have
$$ \ddbar \log |\tilde \sigma'|^2_{\tilde h'} +  \chi \geq 0,$$
therefore  $$\eta_\epsilon \in PSH(X, \chi) \cap C^0(X). $$
In particular, for sufficiently small $\epsilon>0$, we have $$ \log \epsilon \leq \eta_\epsilon \leq 0. $$
We define $\rho_\epsilon$ by
$$\rho_\epsilon = F\left( \frac{100\eta_\epsilon}{\log \epsilon} \right). $$

The following estimate is based on the calculations in \cite{S2} (see ).
\begin{lemma} \label{cutoff1}  There exists $C>0$ such that for any $t\in (0,1]$ and  any $0<\epsilon<1$, we have
$$\int_{  X} |\nabla \rho_\epsilon|^2 \wedge  \omega(t)^n \leq C(-\log \epsilon)^{-1}  t^{n-\kappa}. $$

\end{lemma}

\begin{proof} There exist $C_1, C_2 >0$ such that
\begin{eqnarray*}
&& \sqrt{-1} \int_X \partial \rho_\epsilon \wedge \dbar \rho_\epsilon \wedge  \omega^{n-1}\\
&=& 10000(\log \epsilon)^{-2}\sqrt{-1} \int_X (F')^2 \partial \eta_\epsilon \wedge \dbar \eta_\epsilon \wedge  \omega^{n-1}\\
&\leq & C_1  (\log \epsilon)^{-2} \int_X  (- \eta_\epsilon) \ddbar \eta_\epsilon \wedge  \omega^{n-1}\\
&=& C_1  (\log \epsilon)^{-2} \left( \int_X  (- \eta_\epsilon)(\chi + \ddbar \eta_\epsilon) \wedge  \omega^{n-1} +   \int_X \eta_\epsilon \chi \wedge   \omega^{n-1} \right) \\
&\leq & C_1  (-\log \epsilon)^{-1}   \int_X (\chi + \ddbar \eta_\epsilon) \wedge  \omega^{n-1} \\
&= & C_1  (-\log \epsilon)^{-1}   \int_X \chi  \wedge  \omega^{n-1} \\
%
%
&\leq &C_2 (-\log \epsilon)^{-1} t^{n-\kappa }.
\end{eqnarray*}

\end{proof}

We will pick one of the level set of $|\tilde\sigma'|^2_h$ to be the hypersurface $E$ in Lemma \ref{gromov}.

\begin{lemma} \label{bdyarea}  There exists $C>0$ such that for any $0<\epsilon_0<1$ and any $t\in (0,1]$, there exists $\epsilon_0^2 \leq   \epsilon \leq \epsilon_0$, such that 
$$ Vol_{ \omega(t)}(\{|\tilde\sigma'|^{200}_{\tilde h'} =   \epsilon \}) \leq C(-\log \epsilon)^{-1/2}   t^{n-\kappa}  . $$

\end{lemma}

\begin{proof} We apply the co-area formula
$$\int_X H dg = \int_{-\infty}^\infty \int_{\{G=u\}} \frac{ H}{|\nabla G|} dg|_{G=u} d u $$
by letting $H= |\nabla G|$ and $G= \rho_\epsilon$.   Applying  the previous lemma, there exists $C>0$ such that for all sufficiently small $\epsilon>0$ and $t\in (0, 1]$,
$$\int_X |\nabla \rho_\epsilon|  \omega^n \leq \left(\int_X |\nabla \rho_\epsilon|^2  \omega^n \right)^{1/2} \left(\int_X \omega^n \right)^{1/2} \leq C(-\log\epsilon)^{-1/2} t^{n-\kappa}. $$
We consider the region $$B_{\epsilon_0} = \{ \epsilon_0^2\leq |\sigma|^{200}_h \leq \epsilon_0  \}.$$ In $B_{\epsilon_0}$,
$$1\leq \rho_{\epsilon_0}=F(100\eta_{\epsilon_0}/\log \epsilon_0) \leq 2$$
and
$$\int_1^2
Vol(\{\rho_{\epsilon_0} = u\})  d u \leq C(-\log\epsilon_0)^{-1/2} t^{n-\kappa} .$$
By mean value theorem, there is $a\in[1,2]$ such that
$$Vol(\{\rho_{\epsilon_0} = a\})   \leq C(-\log\epsilon_0)^{-1/2} t^{n-\kappa} .$$
In other words,
$$Vol(\{\Phi^*|\sigma|_h^{200}=\epsilon_0^{3-a}\})   \leq C(-\log\epsilon_0)^{-1/2} t^{n-\kappa} .$$
\end{proof}

Let 
\begin{equation}\label{deset}
D_\epsilon = X \setminus \{ |\tilde\sigma'|^{200}_{\tilde h'} <\epsilon\}. 
\end{equation}
For sufficiently large $N>0$
\begin{equation}\label{gradchi}
\sup_X \left |\partial |\tilde \sigma'|_{\tilde h'} ^{2N}\right|_{\chi}  <\infty
\end{equation}
 because there exists $C=C(N)>0$ such that
$$ \chi\geq |\tilde \sigma' |^{2N}_{\tilde h'} \geq C g_A$$
for some fixed K\"ahelr metric $g_A$ on $X$. Without loss of generality, we can assume $N=100$ for simplicity.
The previous lemma indicates that for almost every sufficiently small $\epsilon>0$, $\partial D_\epsilon$ has very small volume. The following lemma also shows that $\{ |\tilde \sigma'|^{200}_h < \epsilon\}$ has very small volume.

\begin{lemma} \label{smvol} For any $\delta>0$, there exists $\epsilon>0$ such that for all $t\in (0,1]$,
$$\int_{|\tilde\sigma'|^{200}_{\tilde h'} \leq \epsilon}  \omega(t)^n \leq \delta t^{n-\kappa}.$$

\end{lemma}

\begin{proof} First we notice that $\rho_\epsilon \geq 1$ when $|\sigma|^{200}_h \leq \epsilon$ and so
$$\int_{|\tilde\sigma'|^{200}_{\tilde h'}  \leq \epsilon}  \omega(t)^n \leq \int_X \rho_\epsilon  \omega(t)^n. $$
Also $$\lim_{\epsilon\rightarrow 0} \int_{|\tilde\sigma'|^{200}_{\tilde h'}  \leq \epsilon} \Omega = 0 $$
and if we let $\theta(t) = \chi+ t\omega_A$,
$$ \int_X \rho_\epsilon \theta^n \leq t^{n-\kappa} \int_X \rho_\epsilon\Omega$$
and
$$\int_X \rho_\epsilon ( \omega(t)^n - \theta^n) = \sum_{l=0}^{n-1} \int_X \rho_\epsilon \ddbar \psi \wedge  \omega(t)^l \wedge \theta^{n-1-l}.$$
Now by similar calculations as in the proof of Lemma \ref{cutoff1}, there exist $C_1, C_2, C_3, C_4>0$ such that

\begin{eqnarray*}
&& \int_X \rho_\epsilon \ddbar \psi\wedge    \omega \wedge \theta^{n-1-l}\\
&=& \int_X \psi \ddbar \rho_\epsilon \wedge  \omega \wedge \theta^{n-1-l} \\
&=&  \int_X \psi  \left( 10^2(\log \epsilon)^{-1}F' \ddbar \eta_\epsilon + 10^4( \log \epsilon)^{-2}F'' \partial \eta_\epsilon \wedge \dbar \eta_\epsilon \right)   \wedge  \omega \wedge \theta^{n-1-l} \\
&\leq& C_1(-\log \epsilon)^{-1} \int_X (\ddbar \eta_\epsilon + \chi)    \wedge  \omega \wedge \theta^{n-1-l} + C_1(-\log \epsilon)^{-1} \int_X    \chi \wedge  \omega \wedge \theta^{n-1-l} \\
&&+C_1  (\log\epsilon)^{-2}\int_X   \partial \eta_\epsilon \wedge \dbar \eta_\epsilon   \wedge  \omega \wedge \theta^{n-1-l} \\
&\leq& C_2(-\log\epsilon)^{-1} [\chi]\wedge[ \omega]^{n-1}-  C_1  (\log\epsilon)^{-2}\int_X  \eta_\epsilon \ddbar \eta_\epsilon   \wedge  \omega \wedge \theta^{n-1-l}\\
&\leq& C_3 (-\log\epsilon)^{-1} [\chi]\cdot [ \omega(t)]^{n-1}\\
&\leq& C_4(-\log\epsilon)^{-1}  t^{n-\kappa}.
\end{eqnarray*}

The lemma easily follows by combining the above estimates.

\end{proof}

Recall that there exists  $L>0$ such that for all $t\in (0,1]$,
$$ diam(X,  \omega(t)) \leq L. $$

\begin{lemma} \label{distance1} For any $\delta>0$, there exists $0<\epsilon <\delta$ such that for any  $t\in (0,1]$ and any two points
$p_1, p_2 \in D_\delta$, 
 there exists  a smooth path $\gamma_t \subset D_ \epsilon $ joining $p_1$ and $p_2$ satisfying
$$\mathcal{L}_{g(t)}(\gamma_t)   \leq d_{g(t)}(p_1, p_2) + \delta. $$
where $\mathcal{L}_{g(t)}(\gamma_t)  $ is the arc length of $\gamma_t$ with respect to the metric $g(t)$.
\end{lemma}

\begin{proof} 
Let
$$K_{\epsilon_1}(t)=\{ x\in X ~|~ d_{g(t)}(x, \partial D_\delta)< \epsilon_1\} . $$
For any $x\in K_{\epsilon_1}(t)$, there exists $x'\in \partial \tilde D_\delta$ such that
\begin{eqnarray*}
|\tilde \sigma|_{\tilde h}^{200}(x) &\geq& |\tilde \sigma|_{\tilde h}^{200}(x') - \left(\sup_X \left|\nabla |\sigma|_{\tilde h'}^{200}\right|_{g(t)}\right) d_{g(t)}(x, x') \\
&\geq& |\tilde \sigma|_{\tilde h}^{200}(x') - C_1 \left(\sup_X \left|\nabla |\sigma|_{\tilde h'}^{200}\right|_{\chi}\right) d_{g(t)}(x, x')\\
&\geq& \delta- C_2 \epsilon_1
\end{eqnarray*}
by Lemma \ref{schwa} and (\ref{gradchi}), where $C_1, C_2$ do not depend on $\delta$, $\epsilon_1$ or $t$. By choosing $\epsilon_1<<\delta$, we have for all $t\in (0,1)$
$$K_{\epsilon_1} (t)\subset D_{\epsilon_1}.$$
We  choose $\epsilon< \epsilon_1$ with 
$$K_{\epsilon_1} (t) \subset D_{\epsilon_1} \subset D_\epsilon$$ and by Lemma \ref{bdyarea}, 
and there exists $C_3>0$ such that 
$$Vol_{g(t)}(\partial  D_\epsilon ) \leq C_3 (-\log \epsilon)^{-\frac{1}{2}} t^{n-\kappa} $$
for all $t\in (0,1)$.

By Corollary \ref{volr}, there exists $c>0$ such that for all $t\in (0,1)$ and $r<1$
$$Vol_{  g (t)} (B_{  g (t)}(p_i, r) \geq  c_1 r^{2n}  t^{n-\kappa} $$
and so 
$$Vol_{  g (t)} (B_{  g (t)}(p_i, \epsilon_1)  \geq c_1 (\epsilon_1)^{2n} t^{n-\kappa}, ~ i=1, 2.$$

Since $B_{  g (t)}(p_i, \epsilon_1) \subset\subset   D_\epsilon$, $i=1, 2$, we can apply Lemma \ref{gromov} by choosing sufficiently small $\epsilon>0$ with 
$$ (-\log \epsilon)^{\frac{1}{2}}  (\epsilon_1)^{2n} >>1$$
and letting $E = D_\epsilon$ and $r=\epsilon_1$. 
Hence  for all $t\in (0, 1)$, there exist $q\in B_{g(t)}(p_2, \epsilon)$ and a minimal geodesic $\hat\gamma_t \subset D_\epsilon$ (with respect to $g(t)$) joining $p_1$ and  $q$ and
$$ \mathcal{L}_{g(t)}(\hat\gamma_t) =d_{g(t)}(p_1, q) \leq d_{g(t)} (p_1, p_2) + \epsilon.$$
Now we can complete the proof of the lemma by letting $\gamma_t$ be the curve combining $\hat\gamma_t$ and a minimal geodesic joining $p_2$ and $q$ because
$$\mathcal{L}_{g(t)}(\gamma_t) \leq d_{g(t)}(p_1, p_2) + 2\epsilon  \leq d_{g(t)}(p_1, p_2) + 2\delta. $$

\end{proof}

We will also need the $C^0$ regularity of metrics. Let $\omega_{can} = \chi+\ddbar \psi_{can}$ be the twisted K\"ahler-Einstein metric on $X_{can}$
$$Ric(\omega_{can})= -\omega_{can} + \omega_{WP}. $$
We now define a semi-flat closed $(1,1)$-current on $X$ introduced in \cite{ST1, ST2} by the following
$$\omega_{SF} = \omega_A + \ddbar \phi_{SF}$$ such that for any $z \in X_{can}^{\circ}$,
$$Ric(\omega_{SF} |_{X_z}) =0, ~~ \int_{X_z} \phi_{SF} \omega_0^{n-\kappa} |_{X_z}=0. $$

The following lemma is due to \cite{FGS, TWY}.

\begin{lemma} \label{equivalence} For any $\epsilon>0$, there exists $h_\epsilon (t)\geq 0$ with $\lim_{t\rightarrow 0} h_\epsilon(t) =0$ such that for all $t\in (0,1] $, we have on $D_\epsilon$,
$$ (1-h_\epsilon(t)) (\Phi^*\omega_{can} + t \omega_{SF}) \leq        \omega(t) \leq (1+h_\epsilon(t)) (\Phi^*\omega_{can} + t \omega_{SF}) .$$
In particular, $\omega_t$ converges in $C^0$-topology on $\Phi^{-1}(X_{can}^\circ)$ to $\Phi^*\omega_{can}$ as $t\rightarrow 0^+$.

\end{lemma}

%
%


%


Now we are ready to prove the main result of this section.

\begin{proposition} \label{convex1} For any $\delta>0$, there exist $\epsilon_1 > \epsilon_2 >0$ such that for all $t\in (0,1]$,

\begin{enumerate}

\item $Vol( X\setminus D_{\epsilon_1},  \omega(t)) < \delta$.

\item for any two points $p, q \in D_{ \epsilon _1}$, there exists a continuous path $\gamma_t \subset D_{\epsilon _2}$ joining $p$ and $q$ such that
$$\mathcal{L}_{g(t)}(\gamma_t)   \leq d_{ \omega(t)}(p, q) + \delta \leq L+\delta.$$

\end{enumerate}

\end{proposition}

\begin{proof}   For any $\delta>0$, we can always choose $\epsilon_1$ sufficiently small so that the first estimate in the proposition holds by Lemma \ref{smvol}.  The second statement follows directly by Lemma \ref{distance1} after choosing both $\epsilon_1$ and $\epsilon_2$ sufficiently small.
\end{proof}

Proposition \ref{convex1} has many geometric consequences. We will use Proposition \ref{convex1} to prove  Theorem \ref{main2}. The following proposition proves the first three statements in Theorem \ref{main2}.

\begin{proposition} \label{contpro}The following hold.
\begin{enumerate}

\item $(X, g(t))$ converges in Gromov-Hausdorff topology to a compact metric space $(\mathcal{Z}, d_\mathcal{Z})$ as $t\rightarrow 0^+$,


\item  $\omega_t$ converges in $C^0$-topology on $\Phi^{-1}(X_{can}^\circ)$ to the pullback of a smooth K\"ahler metric $g_{can}$ on $X_{can}^\circ$ as $t\rightarrow 0^+$,

\item the metric completion of $(X_{can}^\circ, g_{can})$ is isomorphic to $(\mathcal{Z}, d_\mathcal{Z})$.

\end{enumerate}

\end{proposition}

\begin{proof} It is proved in \cite{FGS} that for any sequence $t_j\rightarrow 0$, $(X, g(t_j))$ converges in Gromov-Hausdorff topology to a compact metric space $(\mathcal{Z}, d_{\mathcal{Z}})$  after possibly passing to a subsequence and so (1) follows.

(2) follows directly from Lemma \ref{equivalence}.

In order to prove (3), we again apply the result of  \cite{FGS}   that $(X_{can}^\circ, g_{can})$ can be locally isometrically embedded into $(\mathcal{Z}, d_{\mathcal{Z}})$ as a dense open set in $\mathcal{Z}$. For any two points $p, q \in X_{can}^\circ$, we choose arbitary $p'$ and $q'$  in the fibre $X_p=\Phi^{-1}(p) $ and $X_q=\Phi^{-1}(q)$. By Proposition \ref{convex1}, for any $\delta>0$, there exists $\epsilon >0$ such that for all $t_j $,
$$p', q' \in  D_{\epsilon}  $$
and there exists a continuous path $\gamma_{t_j} \subset D_{\epsilon}$ joining $p'$ and $q'$ such that
$$\mathcal{L}_{g(t_j)}(\gamma_{t_j})  \leq d_{ g(t_j)}(p', q') + \delta. $$
Since the fibre diameter uniformly tends to $0$ away from singular fibres and $g_{t_j}$ converges uniformly in $C^0$ to $g_{can}$ on $D_\epsilon$, $p$ and $g$ must converge in Gromov-Hausdorff distance to $p$ and $q$ on $(\mathcal{Z}, d_\mathcal{Z})$. Then there exists $J>0$ such that for all $j\geq T$, we have 
$$\mathcal{L}_{g(t_j)}(\gamma_{t_j})    \leq d_\mathcal{Z}(p, q) + 2\delta.$$
We can also assume for all $j\geq J$, we have 
$$g_{can} \leq (1+\delta)g_{t_j} $$
on $D_\epsilon$ by the $C^0$-estimate and convergence of $g(t_j)$ to $g_{can}$.
Therefore for $j\geq J$, we have
\begin{eqnarray*}
d_{g_{can}|_{X_{can}^\circ}}(p, q) &\leq&  |\Phi(\gamma_{t_j})|_{g_{can}}  \leq  (1+\delta) \mathcal{L}_{g(t_j)}(\gamma_{t_j})   \\
&\leq & (1+\delta) (d_\mathcal{Z}(p, q) + 2\delta)\\
&\leq& d_\mathcal{Z}(p, q)  +  \delta(diam(\mathcal{Z}, d_\mathcal{Z}) + 2+2\delta)
\end{eqnarray*}
where $d_{g_{can}|_{X_{can}^\circ}}$ is the distance function on $X_{can}^\circ$ induced by $g_{can}|_{X_{can}^\circ}$.  By letting $\delta \rightarrow 0$, we have on $X_{can}^\circ$,
\begin{equation}\label{d1side}
d_{g_{can}|_{X_{can}^\circ}} \leq d_\mathcal{Z}.
\end{equation}
On the other hand, for any $\delta>0$, by definition there exists a continuous path $\gamma_\delta$  in $X_{can}^\circ$ such that
$$ \left| \mathcal{L}_{g_{can}} (\gamma_\delta) -  d_{g_{can}|_{X_{can}^\circ}} (p, q) \right|< \delta, $$
while for sufficiently large $j>0$,
$$\mathcal{L}_{g_{can}} (\gamma_\delta)   \geq  \mathcal{L}_{g(t_j)} (\gamma'_\delta)- \delta \geq d_{ g(t_j)}(p', q') - \delta \geq d_\mathcal{Z}(p, q) - 2\delta,$$
where $\gamma'_\delta$ is a lift of $\gamma_\delta$ with $p'$ and $q'$ as the end points.
Therefore we have on $X_{can}^\circ$,
$$d_{g_{can}|_{X_{can}^\circ}} \geq d_\mathcal{Z}$$
and so combined with (\ref{d1side}), we have $$d_{g_{can}|_{X_{can}^\circ}}  = d_\mathcal{Z}|_{X_{can}}. $$
It immediately implies that the identity map on $X_{can}^\circ$ induces a Lipschitz map
$$\mathcal{F}: (\mathcal{Z}, d_\mathcal{Z}) \rightarrow (Z_{can}, d_{Z_{can}}), $$
where $(Z_{can}, d_{Z_{can}})$ is the metric completion of $(X_{can}^\circ, g_{can})$.

Since $X_{can}^\circ$ is a dense open set in $\mathcal{Z}$, for any point $z\in \mathcal{Z}$, there exist a sequence $z_i \in X_{can}^\circ$ converging to $z$ with respect to $d_\mathcal{Z}$. Since $d_{\omega_{can}|_{X_{can}^\circ}}  = d_\mathcal{Z} $ on $X_{can}^\circ$, $z_j$ must also converge in $(Z, d_Z)$ and so $\mathcal{F}$ must be injective. Same argument implies $\mathcal{F}$ is also surjective. This completes the proof of the proposition.

\end{proof}

We have completed the proof for  the statements (1), (2), (3) of Theorem \ref{main2} by Proposition \ref{contpro}.

Now we will prove the last part of Theorem \ref{main2}. First, we state the following H\"older estimate in \cite{Kol2} for complex Monge-Amp\`ere equations on K\"ahler orbifolds.

\begin{lemma}\label{holder1}Let $X$ be an $n$-dimensional K\"ahler orbifold. Let $\omega_{orb}$ be an orbifold K\"ahler metric on $X$ and $\Omega$ be an orbifold volume form on $X$.  We consider the following complex Monge-Amp\`ere equation
$$(\omega_{orb} + \ddbar\varphi)^n = F \Omega, $$
where $F$ is a non-negative function with $\int_X F \Omega = \int_X \omega_{orb}^n$.
If $||F||_{L^p(X, \Omega)} <\infty$ for some $p>1$, then there exist $\alpha=\alpha(X, p) \in (0,1)$ and $C=C(X, p, \omega_{orb},  ||F||_{L^p(X, \Omega)})>0$ such that
$$||\varphi -\sup_X \varphi||_{C^\alpha(X, \omega_{orb})} \leq C. $$

\end{lemma}

The following lemma is a slight orbifold generalization of a distance estimate obtained in \cite{Li}.

\begin{lemma} \label{holder2} Let $X$ be an $n$-dimensional K\"ahler orbifold. Let $\omega_{orb}$ be an orbifold K\"ahler metric on $X$. Suppose $\omega \in [\omega_{orb}]$ is an orbifold K\"ahler metric on $X$ satisfying
$$\left\|\frac{\omega^n }{\omega_{orb}^n} \right\|_{L^p(X, \omega_{orb})} <\infty. $$
Then there exist $\alpha=\alpha(X, p)\in (0,1)$ and $C=C\left(X, p, \omega_{orb},  \left\|\omega^n /\omega_{orb}^n \right\|_{L^p(X, \omega_{orb})} \right)$ such that
$$ d_{\omega}(p, q) \leq  Cd_{\omega_{orb}}(p, q)^\alpha $$
for any two points $p, q \in X$.
\end{lemma}

\begin{proof} The proof follows the argument of \cite{Li}. We include the argument for the sake of completeness since it is fairly short and effective.

For any $p\in X$ we consider the distance function $f(z)=d_{\omega}(p,z)$. Let $\pi:\tilde B\rightarrow B$ be a local orbifold uniformization at $p$ on a metric ball $B=B_{\omega_{orb}}(p,4r_0)$ for some uniform radius $r_0$. Let $\tilde{\omega}_{orb}=\pi^*\omega_{orb}$.  Denote $\tilde B=\tilde B_{\tilde{orb}}(\tilde{p},4r_0)$ the lifting metric ball and $\tilde f(\tilde z)=d_{\tilde \omega_{orb}}(\tilde p,\tilde z)$ the distance function accordingly.

For any $\tilde{q}\in \tilde B_{\tilde{\omega}_{orb}}(\tilde{p},r_0)$ and radius $r\le r_0$, we define a cut-off function $\tilde\rho_r$ via
$$\tilde\rho_r = F(d_{\tilde\omega_{orb}}(\tilde q, \cdot)/r)$$
where $d_{\tilde\omega_{orb}}$ is the distance function with respect to $\tilde\omega_{orb}$ and $F$ is one smooth nonnegative cut-off function with $F(x)=1$ for $x\in [0, 1]$ and $F(x)=0$ for $x\geq 2.$ Then $-C_1 r^{-2} \tilde \omega_{orb}\leq \ddbar \tilde \rho_r \leq C_1 r^{-2} \tilde \omega_{orb}$ in $\tilde B_{\tilde{\omega}_{orb}}(\tilde{q},2r)\subset \tilde B$ for some fixed $C_1>0$, and for $r<r_0$ there exist $C_2, C_3, C_4 >0 $ such that
\begin{eqnarray*}
\int_{\tilde B_{\tilde\omega_{orb}}(\tilde q, r)} tr_{\tilde \omega_{orb}}(\tilde \omega) \tilde\omega_{orb}^n &\leq & \int_{\tilde B_{\tilde{\omega}_{orb}}(\tilde{q},2r)} \tilde \rho_r tr_{\tilde \omega_{orb}}(\tilde \omega)\tilde \omega_{orb}^n\\
&=&\int_{\tilde B_{\tilde{\omega}_{orb}}(\tilde{q},2r)} \tilde \rho_r \tilde \omega_{orb}^n + n\int_{\tilde B_{\tilde{\omega}_{orb}}(\tilde{q},2r)} \tilde \rho_r \ddbar \tilde \varphi \wedge \tilde \omega_{orb}^{n-1}\\
&\leq& C_2 r^{2n} + n\int_{\tilde B_{\tilde{\omega}_{orb}}(\tilde{q},2r)}(\tilde\varphi -\tilde\varphi(\tilde p)) \ddbar \tilde \rho_r  \wedge \tilde \omega_{orb}^{n-1}\\
&\leq& C_2 r^{2n} + C_3r^{-2+\alpha}\int_{\tilde B_{\tilde\omega_{orb}}(\tilde p, 2r)} \omega_{orb}^n\\
&\leq& C_4 r^{2n-2+\alpha}.
\end{eqnarray*}
Obviously, $|\nabla\tilde f|^2_{\tilde \omega} =1$. Therefore $|\nabla\tilde  f |_{\tilde \omega_{orb}} \leq tr_{\tilde\omega_{orb}}(\tilde\omega).$ It follows that for all $r\leq r_0$, we have
$$\int_{\tilde B_{\tilde \omega_{orb}}( \tilde q, r)} |\nabla \tilde f|^2_{\tilde\omega_{orb}} \tilde \omega_{orb}^n  \leq C_4 r^{2n-2+\alpha}. $$
Then by Morrey's embedding theorem and the fact that $\tilde f(\tilde p)=0$, we have
$$|\tilde f|_{C^{\alpha/2}(B_{\tilde\omega_{orb}}(\tilde p, \beta r_0))} \leq C_5$$
for some fixed $C_5>0$. In particular, we have
$$d_{\tilde{\omega}}(\tilde p,\tilde q)\le C_5 d_{\tilde \omega_{orb}}(\tilde p,\tilde q)^{\alpha/2}$$
for all $\tilde q\in \tilde B_{\tilde \omega_{orb}}(\tilde p,r_0). $
The corresponding distance function on $X$ satisfies
$$d_\omega(p,q)\le d_{\tilde{\omega}}(\tilde p,\tilde q)\le C_5 d_{\tilde \omega_{orb}}(\tilde p,\tilde q)^{\alpha/2} =C_5 d_{\omega_{orb}}(p,q)^{\alpha/2}$$
for all $q\in B_{\omega_{orb}}(p,r_0).$ 
This completes the proof for the lemma.
\end{proof}

\begin{proposition} \label{contpro2} If $X_{can}$ has only orbifold singulairties, then $(\mathcal{Z}, d_\mathcal{Z})$ is homeomorphic to $X_{can}$.

\end{proposition}

\begin{proof} We will break the proof into the following steps.

\noindent {\bf Step 1.}  Let $\omega_{orb}$ be a smooth orbifold K\"ahler metric on $X_{can}$ in the same class of $\omega_{can}$.  Let $F=-\log \left( \frac{\omega_{can}^\kappa}{(\omega_{orb})^\kappa} \right)$. We claim that $F$ is bounded above. For simplicity, we assume $X_{can}$ is smooth. We let $\omega(t)$, $\chi$, $\omega_A$ and $\psi(t)$ be defined as in (\ref{contin1}). We then consider the following quantity
$$H= \log tr_{\omega(t)}(\chi) - B \psi(t) $$
for $t\in (0, 1]$ and some sufficiently large $B>0$ to be determined. Straightforward calculations show that there exist $C_1, C_2>0$ such that for all $t\in (0, 1]$
$$\Delta_t \log H \geq (B-C_1)tr_{\omega(t)}(\chi) - C_2,$$
using the fact that $\psi(t)$ is uniformly bounded in $L^\infty(X)$ for $t\in (0,1]$ and the curvature of $\chi$ is uniformly bounded, where $\Delta_t$ is the Laplace operator associated to $\omega(t)$. After applying the maximum principle, we conclude that $H$ is uniformly bounded above and so there exists $C_3>0$ such that
$$tr_{\omega(t)}(\chi) \leq C_3$$
and we immediately obtain the upper bound for $F$ by letting $t\rightarrow 0$ and using the mean value inequality since $\omega(t)$ converges to $\omega_{can}$ on $X^\circ$ and $\chi$ is equivalent to $\omega_{orb}$.

Let $\Omega$ be the smooth volume form on $X$ such that $$\ddbar\log \Omega = \chi \in [K_X]$$
as before.
If we let $\omega_{can} = \chi + \ddbar \varphi_{can}$,
$$\omega_{can}^\kappa= (\chi + \varphi_{can})^\kappa = e^{\varphi_{can}}\Phi_*\Omega= e^{-F} (\omega_{orb})^\kappa$$
for some $\varphi_{can}\in PSH(X_{can}, \chi)\cap C^0(X_{can})$, where $\Phi_*\Omega$ is the pushforward or fibre-integration of $\Omega$. In particular, $\varphi_{can}$ is smooth on $X_{can}^\circ$.
By \cite{ST2}, there exists $p>1$ such that $$||e^{-F}||_{L^p(X_{can}, \omega_{orb})}<\infty .$$
Also we have$$\ddbar F = - \omega_{can} +\omega_{WP} - Ric(\omega_{orb}) .$$
where the Weil-Petersson current $\omega_{WP}$ is a  $(1,1)$-current from the variation of the complex structure of smooth Calabi-Yau fibres. Since $\omega_{WP}$ is semi-positive on $X_{can}^\circ$, we have
$$Ric(\omega_{orb})+\omega_{can}+\ddbar F \geq 0$$
 on $X_{can}^\circ$ and $\omega_{can}$ has bounded local potentials. This implies that $F$ is quasi-plurisubharmonic on $X_{can}^\circ$ with respect to $Ric(\omega_{orb})+\omega_{can}$, by extension theorem of quasi-plurisubharmonic function, $F$ must be also quasi-plurisubharmonic on $X_{can}$ since $F$ is bounded above and $X_{can}\setminus X_{can}^\circ$ is an analytic subvariety of $X_{can}$, i.e., 
$$F\in PSH(X_{can}, Ric(\omega_{orb})+\omega_{can}).$$ 
In particular, it implies that $\omega_{WP}$ is a global semi-positive current on $X_{can}$. In fact, the semi-positivity of $\omega_{WP}$ holds in general for any canonical models (\cite{GS}). By the standard approximation theory for plurisubharmonic functions, there exist a sequence orbifold smooth functions. $F_j \in PSH(X_{can}, B\omega_{orb}+\omega_{can})$ for some fixed constant $B>0$ such that $F_j$ converges to $F$ decreasingly. Then immediately, we have
$$||e^{-F_j}||_{L^p(X_{can}, \omega_{orb})}\leq ||e^{-F}||_{L^p(X_{can}, \omega_{orb})}<\infty. $$
We then consider the following complex Monge-Amp\`ere equations
$$(\omega_{orb}+\ddbar \varphi_j)^\kappa = c_j e^{-F_j} (\omega_{orb})^\kappa, $$
where $c_j$ is the norming constant with
$$[\omega_{orb}]^\kappa = c_j \int_{X_{can}} e^{-F_j} (\omega_{orb})^\kappa. $$
$c_j$ is uniformly bounded below because $F_j$ decreases to $F_j$ and it is also uniformly bounded above because $F_j$ are uniformly bounded from above.  Therefore 
$$||c_j e^{-F_j}||_{L^p(X_{can}, \omega_{orb})}$$ is uniformly bounded above for all $j$.

\noindent {\bf Step 2.} By Lemma \ref{holder1}, there exist $\alpha$ and  $C_1>0$ such that for all $j$,
$$||\varphi_j||_{C^\alpha (X_{can}, \omega_{orb})} \leq C_1. $$
and by Lemma \ref{holder2} that there exists $C_2$ such that on $X$,
$$d_{g_j}(\cdot, \cdot)  \leq C_2 d_{g_{orb}}(\cdot, \cdot)^\alpha.$$
We then apply the maximum principle to the following quantity for sufficiently large $K>0$
$$H = \log tr_{\omega_j}(\omega_{orb}) - K \varphi_j,$$
where $\omega_j = \omega_{orb}+ \varphi_j$. 
Then similarly to the Schwarz lemma, there exists $C_3>0$ such that for all $j$,
$$tr_{\omega_j}(\omega_{orb})\leq C_3, $$
or equivalently,
$$\omega_j \geq (nC_3)^{-1} \omega_{orb}$$
since the curvature of $\omega_{orb}$ is bounded and $\varphi_j$ is uniformly bounded in $L^\infty$ for all $j$.

In particular, this implies that there exists $C_4, C_5>0$ such that
$$Ric(\omega_j) = \ddbar F_j +Ric(\omega_{orb}) \geq - C_4 \omega_{orb} \geq -  C_5\omega_j.$$
Therefore the Ricci curvature of $g_j$ are uniformly bounded below. Let $g_j$ and $g_{orb}$ be the orbifold K\"ahler metrics associated to $\omega_j$ and $\omega_{orb}$. We can apply the result of \cite{FGS} and there exists $D>0$ such that
$$diam_{g_j}(X_{can}) \leq D. $$

\noindent {\bf Step 3.} Let $\sigma$ be an effective $\mathbb{Q}$-divisor in the class $[\omega_{can}]$ such that $X_{can}\setminus X_{can}^\circ$ is contained in the support of $\sigma$ and $h$ is the smooth hermitian metric on $[\sigma]$ with $Ric(h) = \omega_{orb}$. We define $D_\epsilon = \{ |\sigma|^2_h < \epsilon \}.$  The same argument in   Proposition \ref{convex1} implies that for for any $\delta>0$, there exist $\epsilon_1 > \epsilon_2 >0$ such that for all $j\geq 0$,

\begin{enumerate}

\item $Vol( X\setminus D_{\epsilon_1}, g_j) < \delta$.

\item for any two points $p, q \in D_{ \epsilon _1}$, there exists a continuous path $\gamma_j \subset D_{\epsilon _2}$ joining $p$ and $q$ such that
$$\mathcal{L}_{g_j}(\gamma_j)\leq d_{g_j}(p, q) + \delta \leq D+\delta. $$

\end{enumerate}

We apply  Lemma \ref{holder2} again. There exist $\alpha\in (0,1)$ and $C_\alpha>0$ such that for any $j$ and any two points $p, q\in X_{can}$
$$d_{g_j}(p, q) \leq C_\alpha d_{g_{orb}}(p, q)^\alpha. $$
After letting $j\rightarrow \infty$, $g_j$ converges to $g_{can}$ smoothly on $X_{can}^\circ$
and for any $p, q\in X_{can}^\circ$, we have
$$d_{g_{can}|_{X_{can}^\circ}}(p, q) \leq C_\alpha d_{g_{orb}}(p, q)^\alpha, $$
where $d_{g_{can}|_{X_{can}^\circ}}$ is the distance function on $X_{can}^\circ$ induced by $g_{can}$.

Since $g_{can}\geq cg_{orb}$ for some $c>0$ on $X_{can}^\circ$ and $(\mathcal{Z}, d_\mathcal{Z})$ is the metric completion of $(X_{can}^\circ, \omega_{can})$,   the local isometry map from $X_{can}^\circ$ into $\mathcal{Z}$ extends to a surjective Lipschitz map
$$\mathcal{F}: (\mathcal{Z}, d_\mathcal{Z}) \rightarrow (X_{can}, g_{orb}).$$
We claim that $\mathcal{F}$ must be injective. For any point $z\in X_{can}$, there exist a sequence $\{z_i\} \in X_{can}^\circ$ converging to $z$ with respect to $g_{orb}$. The inequality
$$d_{g_{can}|_{X_{can}^\circ}}(\cdot, \cdot) \leq C \left(d_{g_{orb}}(\cdot, \cdot)|_{X_{can}^\circ }\right)^\alpha$$ implies that
$\{z_i\}$ is also a Cauchy sequence in $(X_{can}^\circ, g_{can})$. Since $(\mathcal{Z}, d_\mathcal{Z})$ is the metric completion of $(X_{can}^\circ, g_{can})$, $\{z_i\}$ must also be a Cauchy sequence of $(\mathcal{Z}, d_\mathcal{Z})$ and so $\mathcal{F}$ must also be injective. This completes the proof of the proposition.

\end{proof}

\section{Proof of Theorem \ref{main3} }
In this section, we will prove Theorem \ref{main3} by the same argument in the proof of Theorem \ref{main2}.
Let $X$ be an $n$-dimensional projective K\"ahler manifold of $c_1(X)=0$.  Suppose $L$ is a semi-ample line bundle and the linear system $|mL|$ induces a holomorphic map
$$\Phi: X \rightarrow Y$$
for sufficiently large $m$. We assume $0< \dim_{\mathbb{C}}Y=k <n$. The general fibre of $X$ over $Y$ is a smooth Calabi-Yau manifold of dimension $n-\kappa$. Let $Y^\circ$ be the set of all smooth points of $Y$ over which $\Phi$ is submersion.  Let $A$ be an ample line bundle over $X$ then by Yau's theorem, there exists a unique Ricci-flat K\"ahler metric $\omega(t) \in [tA+L]$.
In other words, let $\Omega$ be a smooth volume form on $X$ with
$$\ddbar\log \Omega=0,~ \int_X \Omega =1,$$ then
$$\omega(t)^n = c_t \Omega, $$
where $c(t)$ is the normalization constant defined by $c_t = [tA+L]^n$. In particular, $c_t$ is uniformly bounded above and below away from $0$ for all $t\in (0, 1]$.

Let $\omega_A$ be a fixed K\"ahler form in $[A]$ and $\chi$ be pullback of the Fubini-Study metric on $Y$ from the projective embedding of the linear system $|mL|$. Then we have
$$t^{-(n-\kappa)} (t \omega_A + \chi + \ddbar\varphi(t))^n =c_t\Omega.$$
It is shown in \cite{ST1, ST2} that $\omega(t)= t \omega_A + \chi + \ddbar\varphi(t)$ converges weakly as $t\rightarrow 0$ to the twisted K\"ahler-Einstein metric $\omega_Y$ on $Y$ defined in \cite{ST1, ST2} is given by
$$Ric(\omega_Y) = \omega_{WP},$$
where $\omega_{WP}$ is the Weil-Petersson current from the variation of the complex structures of the Calabiy-Yau fibres over $Y$. More precisely, $\omega_Y = \chi + \ddbar\varphi_Y$ for some $\varphi_Y\in PSH(Y, \chi)\cap C^0(Y)$ and
$$\omega_Y^\kappa= (\chi+ \ddbar \varphi_Y)^\kappa = c_0\Phi_* \Omega,$$
where $\Phi_*\Omega$ is the pushforward of $\Omega$, or the integration over the fibres.  In particular, $\omega_Y$ is smooth on $Y^\circ$. The following proposition is proved by exactly the same argument in the proof of Proposition \ref{contpro}.

\begin{proposition}
Let $\omega_t= t\omega_A + \chi + \ddbar\varphi(t)$ and $g(t)$ be the associated K\"ahler metric on $X$. Then the following hold.
\begin{enumerate}

\item $(X, g_t)$ converges in Gromov-Hausdorff topology to a compact metric space $(\mathcal{Z}, d_\mathcal{Z})$ as $t\rightarrow 0^+$,

\item  $\omega_t$ converges in $C^0$-topology on $\Phi^{-1}(Y^\circ)$ to the pullback of a smooth K\"ahler metric $\omega_Y$ on $Y^\circ$ as $t\rightarrow 0^+$,

\item the metric completion of $(Y^\circ, \omega_{orb})$ is isomorphic to $(\mathcal{Z}, d_\mathcal{Z})$.

\end{enumerate}

\end{proposition}

The same argument in Proposition \ref{contpro2} also gives the following proposition.
\begin{proposition} If $Y$ has only orbifold singulairties, then $(\mathcal{Z}, d_\mathcal{Z})$ is homeomorphic to $X_{can}$.

\end{proposition}

We have now proved Theorem \ref{main3} by combining the above two propositions.

\section{Proof of Theorem \ref{main1} for long time collapsing solutions of the K\"ahler-Ricci flow }

In this section, we will prove Theorem \ref{main1} using the results in Section 2 to derive a uniform diameter estimate for the normalized K\"ahler-Ricci flow on projective manifolds with semi-ample canonical bundle. As before, $X$ is an $n$-dimensional projective manifold with semi-ample $K_X$ and its canonical model $X_{can}$ has dimension $0<\kappa <n$. We will keep the same notations as in section 2. Let $\omega_0$ be any initial K\"ahler form associated to $g_0$.
Then the normalized K\"ahler-Ricci flow (\ref{krflow}) on $X$ %
starting with $g_0$ can be expressed as
 \begin{equation}\label{maflow}
\left\{
\begin{array}{l}
{ \displaystyle \ddt{\varphi} = \log \frac{ e^{(n-\kappa)t} (\chi + e^{-t} (\omega_0 - \chi) +\ddbar \varphi)^n}{\Omega} - \varphi,}\\
\\
\varphi(0) = 0.
\end{array} \right.
\end{equation}
In particular, the K\"ahler class of $[\omega(t)]$ associated to $g(t)$ satisfies
\begin{equation}\label{classcon}
[\omega(t)] = [\chi] + e^{-t}[\omega_0 - \chi]
\end{equation}
for all $t\in [0, \infty)$. 

The following lemma is obtained in \cite{ST1, ST2, SoT4}.
\begin{lemma} \label{0est} There exists $C>0$ such that for all $t\geq 0$, we have
$$||\varphi ||_{L^\infty(X)} + \left|\left| \ddt{\varphi} \right|\right|_{L^\infty(X)}\leq C. $$
Furthermore,  for any compact set $K \subset X^\circ$, $\varphi$ converges in $C^{1, \alpha}$ to $\varphi_{can}$ for any $\alpha \in [0, 1)$.

\end{lemma}

The following lemma  is proved in \cite{TWY} for the local $C^0$-convergence of the evolving collapsing metrics. 
\begin{lemma}   Let $\omega(t)$ be the solutions of the K\"ahler-Ricci flow (\ref{krflow}). Then the following hold.

\begin{enumerate}

\item For any compact subset $K \subset X^\circ$,

$$\lim_{t\rightarrow \infty} ||\omega(t) - \omega_{can}||_{C^0(K, \omega_0)} =0. $$

\item For any compact subset $K' \subset X_{can}^\circ$, the fibre metric over any point in $K'$ converges after rescaling to a Ricci-flat K\"ahler metric uniformly. More precisely, let $\omega_{CY, y}$ be the unique Ricci-flat K\"ahler metric in the K\"ahler class $[\omega_0|_{X_y}$. Then
$$\lim_{t\rightarrow \infty} \sup_{y\in K'} || e^t \omega(t)|_{X_y} - \omega_{CY, y}||_{C^0(X_y, \omega_0|_{X_y}} = 0. $$

\end{enumerate}

\end{lemma}

The uniform bound for the scalar curvature of global solutions of the normalized K\"ahler-Ricci flow is established in \cite{SoT4} as in the following lemma. 
\begin{lemma} The scalar curvature $R$ of $g(t)$ is uniformly bounded, i.e., there exists $C>0$ such that for all $t\geq 0$,

$$\sup_{X} |R(\cdot, t)| \leq C. $$

\end{lemma}

We now state the relative volume comparison established in \cite{TiZ3}.

\begin{lemma} \label{rvolcom}For any $B\geq 1$, there exists $k=k(m, B)>0$ such that the following hold. Let $g(t)$ be a solution to the Ricci flow on a compact $n$-dimensional Riemannian manifold $M$ over time $0\leq t\leq r_0^2$. If
$$|Ric| \leq r_0^{-2}, ~in~ B_{g(0)}(x_0, r_0)\times [0, r_0^2], $$
 then for any $B_{g(r_0^2)}(x, r) \subset B_{g(r_0^2)}(x_0, B r_0)$ satisfying
$$R|_{t=r_0^2}\leq r^{-2}, ~in~ B_{g(r_0^2)}(x, r), $$
we have
$$\frac{Vol(B_{g(r_0^2)}(x, r), g(r_0^2))}{r^m} \geq k \frac{Vol(B_{g(0)}(x_0, r_0), g(r_0^2))}{r_0^m} .$$

\end{lemma}

We let $D_\epsilon$ be the set defined as (\ref{deset}).
\begin{lemma} \label{convex32} Let $g(t)$ be the global solution of the normalized K\"ahler-Ricci flow on $X$. There exists $L>0$ such that for any $\delta>0$ there exist $\epsilon_1, \epsilon_2 >0$ and $T>0$   so that the following hold.
\begin{enumerate}

\item $Vol( X\setminus D_{\epsilon_1}, g(t)) < \delta e^{-(n-\kappa)t}$ for all $t\geq 0$.

\item For any two points $p, q \in D_{\epsilon_1 }$ and $t>T$, there exists a continuous path $\gamma_t \subset D_{ \epsilon_2 }$ joining $p$ and $q$ such that
$$\mathcal{L}_{g(t)}(\gamma_t) \leq  L. $$
\end{enumerate}

\end{lemma}

\begin{proof}  Using the fact that $\varphi(t)$ is uniformly bounded, the same argument in the proof of Proposition \ref{convex1} can show that for any $\delta>0$ there exists $\epsilon_1>0$ such that
$$Vol( X\setminus D_{\epsilon_1}, g(t)) < \delta e^{-(n-\kappa)t} $$ for all $t\geq 0$. Let $p' = \Phi(p)$ and $q'=\Phi(q)$ for any two points $p, q \in X\setminus D_{\epsilon_1}$. Since $ (X_{can}^\circ, g_{can})$ is almost geodesically convex in $(X_{can}, g_{can})$, there exists a continuous path $\gamma'$ in $\Phi(X\setminus D_{\epsilon_2})$ joining $p'$ and $q'$ such that
$$\mathcal{L}_{g_{can}}(\gamma') \leq D+\delta$$
by choosing sufficiently small $\epsilon_2>0$, where $D$ is the diameter of $(X_{can}, \omega_{can})$.
We lift $\gamma'$ to a continuous path $\gamma$ in $ X\setminus D_{\epsilon_1}$ joining $p$ and $q$. Since $\omega(t)$ converges to $\omega_{can}$ on $\Phi^{-1}(X\setminus D_{\epsilon_1})$ to $\omega_{can}$ in $C^0$-topology uniformly for as $t\rightarrow \infty$,  there exists $T>0$ such that
$$\mathcal{L}_{g(t)}(\gamma)  \leq\mathcal{L}_{g_{can}}(\gamma')  + \delta  \leq D+2\delta. $$
This completes the proof of the lemma.

\end{proof}

Now we can prove Theorem \ref{main1} as a consequence of the following proposition.
\begin{proposition} \label{mainprop} If there exists an open domain $U$ in $X$ and $\Lambda>0$ such that such that $\sup_{U\times [0, \infty)} |Ric(g(t))|_{g(t)} <\Lambda $. Then there exists $D>0$ such that for all $t\in [0, \infty)$ we have
$$diam(X, g(t)) < D .$$

\end{proposition}

\begin{proof} By the result in \cite{SoT4}, the scalar curvature is uniformly bounded along the flow and so we can apply Lemma \ref{rvolcom} for  relative volume comparison. We then apply Lemma \ref{convex32} with the same notations. We fixed a base point $P\in X^\circ$. Suppose there exist a point $Q\in X$ and $t>0$ such that
$$2L <d_{\omega(t)}(P, Q) < 4 L. $$
By Lemma \ref{rvolcom} for the relative volume comparison, there exists $k=k(X, g_0)>0$ such that
$$Vol_{\omega(t)}(B_{\omega(t)}(Q, L) \geq k L^{2\kappa} e^{-(n-\kappa)t}. $$
We will apply Lemma \ref{convex32} by choosing
$$\delta = 100^{-1}k L^{2\kappa} $$ and $T>0$ accordingly.
Suppose %
$$\limsup_{t\rightarrow \infty} diam(X, g(t)) = \infty. $$
Then there exist $t'>T$ and $Q'\in X$ such that
$$2L <d_{\omega(t')}(P, Q') < 4 L. $$
Obviously,
$$B_{\omega(t')}(Q', L) \subset X\setminus D_{\epsilon_1}.$$
Then applying Lemma \ref{convex32}  again, we have
$$ k L^{2\kappa} e^{-(n-\kappa)t}\leq Vol_{\omega(t')}(B_{\omega(t')}(Q', L) )\leq \delta e^{-(n-\kappa)t}.$$
Contradiction. \end{proof}

Corollary \ref{mainc1} immediately follows from Proposition \ref{mainprop} and the result proved in \cite{TWY} that the curvature tensor $g(t)$ is uniformly bounded on any compact subset of $X^\circ$ when the general fibre of $\Phi: X \rightarrow X_{can}$ is a complex torus.

\section{Proof of Theorem \ref{KRF: minimal model}}

We will prove Theorem \ref{KRF: minimal model} in this section. Let $X$ be a projective manifold with semi-ample canonical bundle of Kodaira dimension $1$. The canonical model $X_{can}$ of $X$ must be a nonsingular Riemann surface and we let $\Phi: X \rightarrow X_{can}$ be the morphism induced by the pluricanonical system. We assume that the general fibre is a complex torus as in the assumptions of Theorem \ref{KRF: minimal model}.  The set of  critical values of $\Phi$  coincide with $\mathcal{S} = X_{can}\setminus X_{can}^\circ$ and we let
$$\mathcal{S} = \{ p_1, ..., p_N\}. $$
Let $g_{can}$ be the twisted K\"ahler-Einstein metric on $X_{can}^\circ$ induced by the twisted K\"ahler-Einstein current $\omega_{can}$ satisfying 
$$Ric(\omega_{can}) = - \omega_{can} + \omega_{WP}$$
on $X_{can}$. 
For each $p_i\in \mathcal{S}$, we define for $\delta>0$, 
$$B_{i, \delta}= \{ p\in X_{can}~|~ d_\chi(p, p_i) < \delta\}, $$
where $d_\chi$ is the distance function with respect to the smooth K\"ahler metric $\chi$. 
Let $g(t)$ be the solution of the normalized K\"ahler-Ricci flow on $X$ and $$f_{i,\delta} =  \int_{ \Phi^{-1}(B_{i, \delta})}\Omega.$$ Then the following holds as $\Phi^{-1}(B_{i, \delta})$ converges to a single fibre as $\delta \rightarrow 0$.
\begin{lemma} \label{smallvol}

$$
\lim_{\delta\rightarrow 0} f_{i,\delta} = 0. 
$$

\end{lemma} 
We now estimate the volume of $\Phi^{-1}(B_{i, \delta})$.

\begin{lemma} \label{smallvo}There exists $C>0$ such that 

$$Vol_{g(t)} \left( \Phi^{-1}(B_{i, \delta}) \right) \leq C  e^{-(n-1)t} f_{i, \delta}. $$

\end{lemma}

\begin{proof} By the uniform $L^\infty$-estimate of $\varphi$ and $\ddt{\varphi}$ in Lemma \ref{0est}, there exists $C>0$ such that for all $t\geq 0$, we have 
$$\omega(t)^n \leq C e^{-(n-1)t} \Omega,$$
where $\Omega$ is a smooth volume form on $X$.  Then the lemma follows by the following calculations 
$$Vol_{g(t)} \left( \Phi^{-1}(B_{i, \delta}) \right) =  \int_{ \Phi^{-1}(B_{i, \delta})} \omega(t)^n
\leq  C e^{-(n-1)t}\int_{ \Phi^{-1}(B_{i, \delta})}\Omega. $$

\end{proof}

The following lemma follows immediately from Proposition \ref {contpro2} (also see \cite{Zh17})  $X_{can}$ is nonsingular.
\begin{lemma} The metric completion of $(X_{can}^\circ, g_{can})$ is homeomorphic to $X_{can}$. 

\end{lemma}

We define $$\gamma_\delta (t)= \sup_{p, q \in \partial B_{i, \delta}} d_{g(t)}(p, q)$$ as the diameter of $\partial B_{i, \delta}$ with respect to $g(t)$. 

\begin{lemma}\label{bdl} For any $\epsilon>0$, there exist $\delta_0>0$ and $T(\delta)>0$ for $0< \delta <\delta_0$ such that for all $0<\delta<\delta_0$ and $t>T(\delta)$, we have 
$$ \gamma_\delta(t) < \frac{\epsilon}{8} .$$
\end{lemma}
\begin{proof} The lemma follows from the geodesic convexity of $(X_{can}^\circ, g_{can})$ in its metric completion and the fact that $g(t)$ converges uniformly to $g_{can}$ uniformly away from singular fibres.

\end{proof}

The following lemma is the key estimate in this section.

\begin{lemma} \label{sec5key} For any $\epsilon>0$, there exist $\delta>0$ and $T>0$ such that for all $t\geq T$, we have 
$$ diam( \Phi^{-1}(B_{i, \delta}), g(t)) < \epsilon, $$
where the distance is calculated in $(X, g(t))$ and $i=1, 2, ..., N$. 

\end{lemma}

\begin{proof}  %
Let
 $$d_{i,\delta} (t) = \frac{1}{2} \sup_{q \in \Phi^{-1} (B_{i, \delta})} d_{g(t)}\left(q, \partial \left( \Phi^{-1} (B_{i, \delta}) \right)\right),$$
where $d_{g(t)}$ is the distance function with respect to $g(t)$.
Then  one can find a geodesic ball $B(t)$ of radius $d_{i,\delta} (t)$ with respect to the metric $g(t)$ lying entirely in $\Phi^{-1} (B_{i, \delta})$.  By Lemma \ref{smallvo} there exists $C>0$ independent of $t$ and $\delta$ such that  
$$Vol_{g(t)} (B(t)) \leq Vol_{g(t)} (\Phi^{-1}(B_{i,\delta})) \leq C e^{-(n-1)t} f_{i, \delta}.$$

Since the diameter of $g(t)$ is uniformly bounded by Theorem \ref{main1} and the Ricci curvature of $g(t)$ is uniformly bounded away from $\mathcal{S}$ for all $t$, we can apply the relative volume comparison in Lemma \ref{rvolcom}: there exists $c>0$ independent of $t$ and $\delta$ such that 
$$ce^{-(n-1)t} d_{i, \delta}(t)^{2n}  \leq Vol_{g(t)} (B(t)) < C e^{-(n-1)t} f_{i, \delta}.$$
This implies that 
$$ d_{i, \delta}(t) \leq (c^{-1} C f_{i, \delta})^{\frac{1}{2n}} \rightarrow 0 $$
 as $\delta \rightarrow 0$, in other words, for any $\epsilon>0$, there exists $\delta>0$ such that for all $t\geq 0$, we have 
\begin{equation}\label{smball}
d_{i, \delta} (t) < \frac{\epsilon}{8}.
\end{equation}

Let $X_y = \Phi^{-1}(y)$ be the fibre of $\Phi$ over $y\in X_{can}$. For fixed $\delta$ chosen above, we can choose $T>0$ such that for all $t\geq T$, 
\begin{equation}\label{fibresm}
 \sup_{y\notin \cup_{i=1}^N B_{i, \delta}} diam( X_y, g(t)) < \frac{\epsilon}{8}
 \end{equation}
because the fibre diameter uniformly goes to $0$ away from singular fibres.  Combining (\ref{smball}), (\ref{fibresm}) and Lemma \ref{bdl}, we have for all $t\geq T$, 
$$diam(\Phi^{-1}(B_{i, \delta}), g(t)) < \epsilon, $$
where the diameter is computed in the ambient space $(X, g(t))$.

\end{proof}

The following lemma holds since $g(t)$ converges in $C^0$ to $g_{can}$ away from singular fibres and the fibre diameter  with respect to $g(t)$ uniformly tends to zero away from $\Phi^{-1} (B_{i, \delta} )$ as $t\rightarrow \infty$.

\begin{lemma} \label{appraw} 

\begin{equation}
\lim_{\delta \rightarrow 0} \limsup_{t\rightarrow \infty} d_{GH}\big((X\setminus \left( \cup_{i=1}^N \Phi^{-1} (B_{i, \delta} )\right) ,g(t)),~(X_{can}\setminus  \left( \cup_{i=1}^N  B_{i, \delta} \right),g_{can})\big) =0.
\end{equation}

\end{lemma}

Theorem \ref{KRF: minimal model} immediately follows from Lemma \ref{sec5key} and Lemma \ref{appraw} since the metric completion $(X_{can}^\circ, g_{can})$ is homeomorphic to $X_{can}$.

\end{document}